\documentclass{amsart}

\usepackage{amsmath}

\usepackage{amsthm}
\usepackage{amsfonts}
\usepackage{dsfont}

\newcommand{\reals}{\mathbb{R}}

\newcommand{\complex}{\mathbb{C}}
\newcommand{\naturals}{\mathbb{N}}

\newcommand{\angles}[1]{\left\langle #1 \right\rangle}

\newcommand{\Ccom}[3]{\big[[#1,#2],#3\big]}

\newcommand{\paraa}[1]{\big(#1\big)}
\newcommand{\parab}[1]{\Big(#1\Big)}
\newcommand{\parac}[1]{\bigg(#1\bigg)}

\newtheorem{theorem}{Theorem}[section]

\newtheorem{lemma}[theorem]{Lemma}
\newtheorem{proposition}[theorem]{Proposition}

\theoremstyle{definition}
\newtheorem{definition}[theorem]{Definition}
\theoremstyle{remark}
\newtheorem{remark}[theorem]{Remark}
\numberwithin{equation}{section}

\newcommand{\xv}{\vec{x}}

\renewcommand{\mid}{\mathds{1}}
\newcommand{\A}{\mathcal{A}}
\newcommand{\Fh}{\mathfrak{F}_\hbar}
\newcommand{\Ah}{\A_\hbar}
\renewcommand{\dh}{\hat{\partial}}
\newcommand{\dhu}{\dh_u}
\newcommand{\dhone}{\dh_1}
\newcommand{\dhv}{\dh_v}
\newcommand{\dhtwo}{\dh_2}

\newcommand{\Xt}{\tilde{X}}
\renewcommand{\Re}{\operatorname{Re}}
\renewcommand{\Im}{\operatorname{Im}}
\renewcommand{\d}{\partial}
\newcommand{\db}{\bar{\d}}
\newcommand{\Xv}{\vec{X}}
\newcommand{\Nv}{\vec{N}}
\newcommand{\Yv}{\vec{Y}}
\newcommand{\ip}[1]{\langle#1\rangle}
\newcommand{\Ft}{\tilde{F}}
\newcommand{\Sym}{\operatorname{Sym}}
\newcommand{\ket}[1]{\big|#1\rangle}

\newcommand{\ad}{a^\dagger}

\newcommand{\intpart}[1]{\lfloor#1\rfloor}
\newcommand{\V}{\mathcal{V}}
\newcommand{\Ld}{\Lambda^\dagger}
\newcommand{\eps}{\varepsilon}
\newcommand{\D}{\mathcal{D}}
\newcommand{\E}{\mathcal{E}}
\newcommand{\F}{\mathcal{F}}
\newcommand{\G}{\mathcal{G}}

\title{Noncommutative Minimal Surfaces} 

\author{Joakim Arnlind}
\address[Joakim Arnlind]{Dept. of Math.\\
Link\"oping University\\
581 83 Link\"oping\\
Sweden}
\email{joakim.arnlind@liu.se}

\author{Jaigyoung Choe}
\address[Jaigyoung Choe]{%
  Korea Institute for Advanced Study\\
  85 Hoegiro Dongdaemun-Gu\\
  Seoul 130-722\\
  South Korea}
\email{choe@kias.re.kr}

\author{Jens Hoppe}
\address[Jens Hoppe]{
Sogang University\\
Sinsu-Dong, Mapo-Gu\\
Seoul 121-741\\
South Korea}

\subjclass[2000]{}
\keywords{}

\begin{document}

\begin{abstract}
  We define noncommutative minimal surfaces in the Weyl algebra, and
  give a method to construct them by generalizing the well-known
  Weierstrass-representation.
\end{abstract}

\maketitle

\noindent Given the growing interest in noncommutative spaces, and
zero-mean-curvature surfaces having been known for more
than 250 years, it is rather astonishing that a general theory
of noncommutative minimal surfaces seems to be lacking.
Our note is a modest attempt to fill this gap.

\section{Preliminaries}

\subsection{Poisson algebraic geometry of minimal surfaces}\label{sec:PoissonMinimal}

\noindent Not long ago, it was shown that the geometry of surfaces
(or, in general, almost K\"ahler manifolds) can be expressed via
Poisson brackets of the functions $x^1,\ldots,x^n$ which provide an
isometric embedding into a given ambient manifold
\cite{ah:kahlerpoisson,ahh:multilinear}. In noncommutative geometry,
as well as quantum mechanics, there is an intimate relationship
between an operator corresponding to the (commutative) function
$\{f,g\}$, and the commutator of the operators that correspond to $f$
and $g$. Therefore, obtaining knowledge about geometrical quantities,
as given in the Poisson algebra generated by $x^1,\ldots,x^n$,
provides information about the corresponding noncommutative
geometrical objects, and how to define them.

Assume that $\Sigma$ is a $2$-dimensional manifold, with local
coordinates $u=u^1,v=u^2$, embedded in $\reals^n$ via the embedding
coordinates $x^1(u,v),x^2(u,v),\ldots,x^n(u,v)$, inducing on
$\Sigma$ the metric
\begin{align*}
  g_{ab} = \d_a\xv\cdot\d_b\xv \equiv \sum_{i=1}^n\paraa{\d_ax^i}\paraa{\d_bx^i}
\end{align*}
where $\d_a=\frac{\d}{\d u^a}$. We adopt the convention that indices
$a,b,p,q$ take values in $\{1,2\}$, and $i,j,k,l$ run from $1$ to
$n$. For an arbitrary density $\rho$, one may introduce a Poisson
bracket on $C^\infty(\Sigma)$ via
\begin{align*}
  \{f,h\} = \frac{1}{\rho}\eps^{ab}\paraa{\d_af}\paraa{\d_b h},
\end{align*}
and we define the function $\gamma=\sqrt{g}/\rho$, where $g$ denotes
the determinant of the metric $g_{ab}$. Setting
$\theta^{ab}=\frac{1}{\rho}\eps^{ab}$ (the Poisson bivector) one
notes that
\begin{align}\label{eq:thetathetagamma}
  \theta^{ap}\theta^{bq}g_{pq} =
  \frac{1}{\rho^2}\eps^{ap}\eps^{bq}g_{pq}
  =\frac{g}{\rho^2}g^{ab} = \gamma^2g^{ab}
\end{align}
since $\eps^{ap}\eps^{bq}g_{pq}$ is the cofactor expansion of the
inverse of the metric. The fact that the geometry of the submanifold
$\Sigma$ can be expressed in terms of Poisson brackets follows from
the trivial, but crucial, observation that the projection operator
$\D:T\reals^n\to T\Sigma$ (where one regards $T\Sigma$ as a subspace
of $T\reals^n$) can be written as
\begin{align*}
  \D(X)^i = \frac{1}{\gamma^2}\sum_{j,k=1}^n\{x^i,x^k\}\{x^j,x^k\}X^j
\end{align*}
for $X\in T\reals^n$. Namely, one obtains
\begin{align*}
  \D(X)^i &= \frac{1}{\gamma^2}\sum_{j,k=1}^n
  \theta^{ab}\theta^{pq}(\d_ax^i)(\d_bx^k)(\d_px^j)(\d_qx^k)X^j\\
  &=\frac{1}{\gamma^2}\sum_{j=1}^n\theta^{ab}\theta^{pq}
  g_{bq}(\d_ax^i)(\d_px^j)X^j = 
  \sum_{j=1}^ng^{ap}(\d_ax^i)(\d_px^j)X^j,
\end{align*}
by using \eqref{eq:thetathetagamma}. From this expression one concludes
that $\D^2=\D$ and that $\D(X)=X$ and $\D(N)=0$ if $X\in T\Sigma$ and
$N\in T\Sigma^{\perp}$. 

In this paper, we shall foremost be interested
in the Laplace-Beltrami operator on $\Sigma$, defined as
\begin{align*}
  \Delta(f) = \frac{1}{\sqrt{g}}\d_a\parab{\sqrt{g}g^{ab}\d_bf}.
\end{align*}
\begin{proposition}\label{prop:LaplaceFormulas}
  For $f\in C^{\infty}(\Sigma)$ it holds that 
  \begin{align*}
    &\Delta(f) =
    \gamma^{-1}\sum_{i=1}^n\{\gamma^{-1}\{f,x^i\},x^i\}\\
    &\Delta(f) = \gamma^{-1}\{\gamma^{-1}\{f,u^a\}g_{ab},u^b\}. 
  \end{align*}
\end{proposition}

\begin{proof}
  Let us prove the first formula; the second one is proven in an
  analogous way. One computes that
  \begin{align*}
    \frac{1}{\gamma}&\sum_{i=1}^n
    \theta^{ab}\d_a\paraa{\gamma^{-1}\theta^{pq}(\d_pf)(\d_qx^i)}\d_bx^i\\
    &=\frac{1}{\gamma}\sum_{i=1}^n
    \theta^{ab}\d_a\paraa{\gamma^{-1}\theta^{pq}(\d_pf)(\d_qx^i)(\d_bx^i)}\\
    &=\frac{1}{\sqrt{g}}\d_a\paraa{\gamma^{-1}\eps^{ab}\theta^{pq}g_{bq}(\d_pf)}
    =\frac{1}{\sqrt{g}}\d_a\paraa{\gamma^{-1}\rho\theta^{ab}\theta^{pq}g_{bq}(\d_pf)}\\
    &=\frac{1}{\sqrt{g}}\d_a\paraa{\gamma^{-1}\rho\gamma^2g^{ap}\d_pf}
    =\frac{1}{\sqrt{g}}\d_a\paraa{\sqrt{g}g^{ap}\d_pf}=\Delta(f),
  \end{align*}
  by using \eqref{eq:thetathetagamma}.
\end{proof}

\noindent On a surface, one may always find \emph{conformal
  coordinates}; i.e., coordinates with respect to which the metric
becomes $g_{ab}=\E(u,v)\delta_{ab}$ for some (strictly positive) function
$\E$. Furthermore, if we choose $\rho=1$ (giving $\gamma=\E$), the
second formula in Proposition \ref{prop:LaplaceFormulas} can be
written as
\begin{align*}
  \Delta(f) = \frac{1}{\E}\{\{f,u^a\}\delta_{ab},u^b\}
  =\frac{1}{\E}\{\{f,u\},u\}+\frac{1}{\E}\{\{f,v\},v\}
\end{align*}
if we assume the coordinates $u,v$ to be conformal. For convenience,
we shall also introduce
$\Delta_0(f)=\{\{f,u\},u\}+\{\{f,v\},v\}=\E\Delta(f)$.

Minimal surfaces can be characterized by the fact that their
embedding coordinates $x^1,\ldots,x^n$ are harmonic with respect to
the Laplace operator on the surface; i.e. $\Delta(x^i)=0$ for
$i=1,\ldots,n$. In local conformal coordinates, due to the above
Poisson algebraic formulas, one may formulate this as follows: A
surface $\xv:D\subset\reals^2\to \reals^n$ is minimal if
\begin{align*}
  &\Delta_0(x^i)=\{\{x^i,u\},u\}+\{\{x^i,v\},v\} = 0\text{ for }i=1,\ldots,n\\
  &\xv_u\cdot\xv_u=\xv_v\cdot\xv_v\text{ and }\xv_u\cdot\xv_v = 0
\end{align*}
(where $\xv_u$ and $\xv_v$ denote the partial derivatives of $\xv$
with respect to $u$ and $v$). Note that the above choice of Poisson
bracket implies that $\{u,v\}=1$. In Section \ref{sec:nms} we will, in
analogy with the above formulation, define noncommutative minimal
surfaces in a (noncommutative) algebra with generators $U,V$
satisfying $[U,V]\sim \mid$; the universal algebra with these
properties is commonly known as the \emph{Weyl algebra}.

\subsection{The Weyl algebra and its field of fractions}\label{sec:weylAlgebra}

\noindent As mentioned in the previous section, the Weyl algebra
provides us with a natural setting in which noncommutative minimal
surfaces may be defined. In this section we recall some basic properties of
the Weyl algebra (and its field of fractions), as well as introducing
the notation which shall later be used.

\begin{definition}[Weyl algebra]
  Let $\complex\angles{U,V}$ denote the free
  (associative) unital algebra generated by $U,V$. Furthermore, for
  $\hbar>0$, let $I_\hbar$ denote the two-sided ideal generated
  by the relation
  \begin{align*}
    &UV-VU=i\hbar\mid.
  \end{align*}
  The \emph{Weyl algebra} is defined as
  $\Ah=\complex\angles{U,V}/I_\hbar$.
\end{definition}

\noindent The Weyl algebra can be embedded in a skew field by a
general procedure \cite{o:lineareqnc}. Let us briefly review the
construction for the purposes of this paper.

Consider the Cartesian product $\Ah\times\Ah^{\times}$, i.e. ordered
pairs $(A,B)$ of elements in $A,B\in\Ah$ with $B\neq 0$, which in the
end will correspond to the expression $AB^{-1}$. The Weyl algebra
satisfies the Ore condition; i.e., for each pair of elements $A,B$
there exist $\beta_1,\beta_2\in\Ah$ such that
\begin{align*}
  A\beta_1=B\beta_2
\end{align*}
(see \cite{l:classalgebras,d:weylalgebras} for a proof of this fact and many other
properties of the Weyl algebra). This property allows one to define a
relation on $\Ah\times\Ah^{\times}$. Namely, $(A,B)\sim (C,D)$ if there exist
$\beta_1,\beta_2\in\Ah$ such that
\begin{align*}
  &A\beta_1 = C\beta_2\\
  &B\beta_1 = D\beta_2,
\end{align*}
and it is straightforward to check that $\sim$ is an equivalence
relation. The quotient $(\Ah\times\Ah^{\times})/\sim$ is denoted by $\Fh$.
Addition in $\Fh$ is defined as follows: Let $\beta_1,\beta_2\in\Ah$
be such that $B\beta_1=D\beta_2$. Then one sets
\begin{align*}
  (A,B)+(C,D) = (A\beta_1+C\beta_2,B\beta_1).
\end{align*}
Likewise, when $\alpha_1,\alpha_2\in\Ah$ are such that
$B\alpha_1=C\alpha_2$, one defines
\begin{align*}
  (A,B)(C,D) = (A\alpha_1,D\alpha_2).
\end{align*}
It is straightforward (although tedious) to check that these are
well-defined operations in $\Fh$ (i.e. they respect equivalence
classes) and that they do not depend on the particular choice of
$\beta_1,\beta_2,\alpha_1,\alpha_2$. Furthermore, both operations are
associative, and they satisfy the distributive law. The unit element
can be represented by $(\mid,\mid)$ and the zero element by
$(0,\mid)$. For every element $A\in\Ah$ we identify $A$ with
$(A,\mid)$ and $A^{-1}$ with $(\mid,A)$ (for $A\neq 0$), and with this
notation it follows that $AB^{-1}=(A,\mid)(\mid,B)=(A,B)$. One easily
checks that $(B,A)$ is the (right and left) inverse of $(A,B)$ and
that $(AB)^{-1}=B^{-1}A^{-1}$. Moreover, if $[A,B]=0$ it holds that
$AB^{-1}=B^{-1}A$, i.e. $(A,B)=(\mid,B)(A,\mid)$.

The Weyl algebra becomes a $\ast$-algebra upon setting $U^\ast=U$ and
$V^\ast=V$, and as a consequence of the universal property of the
fraction ring, the $\ast$-operation can be extended to $\Fh$. Thus,
$\Fh$ is a $\ast$-algebra, and it follows that
$(A,\mid)^\ast=(A^\ast,\mid)$ and $(\mid,A)^\ast=(\mid,A^\ast)$ (where
the last equality can be written as $(A^{-1})^\ast=(A^\ast)^{-1}$) for
all $A\in\Ah$. Hence, it holds that
\begin{align*}
  (AB^{-1})^\ast=(A,B)^\ast = \paraa{(A,\mid)(\mid,B)}^\ast 
  = (\mid,B^\ast)(A^\ast,\mid) = (B^\ast)^{-1}A^\ast.
\end{align*}
In the following, we shall drop the (cumbersome) notation $(A,B)$ and
simply write $AB^{-1}$; moreover, we do not distinguish between an
element $A\in\Ah$ and its corresponding image in $\Fh$. An element
$A\in\Fh$ is called \emph{hermitian} if $A^\ast=A$. The real and
imaginary parts of an element are defined as
\begin{align*}
  &\Re(A) = \frac{1}{2}\paraa{A+A^\ast}\\
  &\Im(A) = \frac{1}{2i}\paraa{A-A^\ast},
\end{align*}
and it is convenient to introduce the notation $U^1=U$ and $U^2=V$,
as well as the derivations
\begin{align*}
  &\dhu(A) \equiv \dhone(A) = \frac{1}{i\hbar}[A,V]\\
  &\dhv(A) \equiv \dhtwo(A) = -\frac{1}{i\hbar}[A,U].
\end{align*}

\begin{proposition}\label{prop:dudvProperties}
  For $A\in\Fh$ and $p(x)\in\complex[x]$ it holds that
  \begin{enumerate}
  \item $\dh_a A^{-1}=-A^{-1}\dh_a(A)A^{-1}$,
  \item $\dh_a\paraa{\dh_b(A)}=\dh_b\paraa{\dh_a(A)}$,
  \item $\dh_a\,p(U^a) = p'(U^a)$\quad\text{(no sum over $a$)},
  \end{enumerate}
  for $a,b=1,2$, where $p'(x)$ denotes the derivative (w.r.t. $x$) of
  $p(x)$.
\end{proposition}

\begin{proof}
  The first property is an immediate consequence of the fact that
  $\dh_a(AA^{-1})=\dh_a(\mid)=0$. For the third property, one computes
  \begin{align*}
    \dh_up(U) = \frac{1}{i\hbar}\sum_{k=0}^n[a_kU^{k},V]
    =\sum_{k=1}^nka_kU^{k-1} = p'(U),
  \end{align*}
  and similarly for $p(V)$.  Finally, to show that the derivatives
  commute, one simply calculates
  \begin{align*}
    \dhu\paraa{\dhv(A)} = \frac{1}{\hbar^2}\Ccom{A}{U}{V}
    =-\frac{1}{\hbar^2}\Ccom{U}{V}{A}-\frac{1}{\hbar^2}\Ccom{V}{A}{U}.
  \end{align*}
  Since $[U,V]=i\hbar\mid$ (and, hence, is in the center of the
  algebra) it follows that
  \begin{align*}
    \dhu\paraa{\dhv(A)} = \frac{1}{\hbar^2}\Ccom{A}{V}{U}
    =\dhv\paraa{\dhu(A)},
  \end{align*}
  which proves the statement.
\end{proof}

\noindent Furthermore, let us introduce $\Lambda=U+iV$ together with
the operators
\begin{align*}
  &\d(A) = \frac{1}{2}\paraa{\dhu(A)-i\dhv(A)}
  =\frac{1}{2\hbar}[A,\Lambda^\ast]\\
  &\db(A) = \frac{1}{2}\paraa{\dhu(A)+i\dhv(A)}
  =-\frac{1}{2\hbar}[A,\Lambda],
\end{align*}
and it follows from Proposition \ref{prop:dudvProperties} that
$\d\db A=\db\d A$. It is useful to note that $[\Lambda,\Lambda^\ast]=2\hbar\mid$.

\begin{definition}
  An element $A\in\Fh$ is called
  \emph{r-holomorphic}\footnote{``rational''-holomorphic} if $\db
  A=0$. An r-holomorphic element $A$ is called \emph{holomorphic} if
  $A\in\Ah$.
\end{definition}

\noindent By $\complex[\Lambda]$ we denote the subalgebra of $\Fh$
generated by $\Lambda$ and $\mid$. It turns out that
r-holomorphic elements can be characterized as elements of
$\complex[\Lambda]$ and their quotients. 

\begin{lemma}\label{lemma:dbHolomorphic}
  An element $A\in\Fh$ is holomorphic if and only if
  $A\in\complex[\Lambda]$.
\end{lemma}

\begin{proof}
  Clearly, if $A\in\complex[\Lambda]$ then $A\in\Ah$ and $\db
  A=-\frac{1}{2\hbar}[A,\Lambda]=0$. Now, assume that $\db A=0$ and
  that $A\in\Ah$. Every element $A\in\Ah$ can be
  written in the following normal form
  \begin{align*}
    A = \sum_{k,l\geq 0}a_{kl}\Lambda^k(\Lambda^\ast)^l,
  \end{align*}
  and one computes
  \begin{align*}
    \db A = \sum_{k\geq 0,l\geq 1}la_{kl}\Lambda^k(\Lambda^\ast)^{l-1}.
  \end{align*}
  The fact that $\db A=0$ implies that $a_{kl}=0$ for $l\geq 1$, which
  implies that $A$ is a polynomial in $\Lambda$. Hence, $A\in\complex[\Lambda]$.
\end{proof} 

\begin{proposition}\label{prop:dbPseudoHolomorphic}
  An element $A\in\Fh$ is r-holomorphic if and only if there
  exist $B,C\in\complex[\Lambda]$ such that $A=BC^{-1}$.
\end{proposition}

\begin{proof}
  Clearly, if $A=BC^{-1}$ with $B,C\in\complex[\Lambda]$, then 
  \begin{align*}
    \db A = (\db B)C^{-1}-BC^{-1}\paraa{\db C}C^{-1} = 0,
  \end{align*}
  by Lemma \ref{lemma:dbHolomorphic}. Now, assume that $A=BC^{-1}$,
  with $B\neq 0$, and that $\db A=0$. From the above equation it follows
  that
  \begin{align*}
    \db B = BC^{-1}\paraa{\db C},
  \end{align*}
  and if $\db C=0$ then $\db B=0$ and Lemma \ref{lemma:dbHolomorphic}
  implies that $B,C\in\complex[\Lambda]$. If $\db C\neq 0$ then $\db
  B\neq 0$ and one obtains
  \begin{align*}
    \paraa{\db B}(\db C)^{-1} = BC^{-1} = A.
  \end{align*}
  It follows that $\db (\db B)(\db C)^{-1}=\db A = 0$, and one may
  repeat the argument with respect to the representation $A=(\db
  B)(\db C)^{-1}$. Thus, as long as $\db^nC\neq 0$ (and, hence, $\db^n
  B\neq 0$) one obtains
  \begin{align*}
    A = (\db^n B)(\db^n C)^{-1}.
  \end{align*}
  For every non-zero $B\in\Ah$ there exists an integer $n_0$ such that
  $\db^{n_0}B\neq 0$ and $\db^{n_0+1}B=0$, since $B$ can be written as
  a polynomial in $\Lambda$ and $\Lambda^\ast$. The above argument
  implies that one can always find $\tilde{B}$ ($=\db^{n_0}B$)
  and $\tilde{C}$ ($=\db^{n_0}C$) such that
  $A=\tilde{B}\tilde{C}^{-1}$, fulfilling
  $\db\tilde{B}=\db\tilde{C}=0$. From Lemma \ref{lemma:dbHolomorphic}
  it follows that $\tilde{B},\tilde{C}\in\complex[\Lambda]$.
\end{proof}

\noindent Note that r-holomorphic elements are the analogues of meromorphic
functions in complex analysis. However, since there is no immediate
concept of point in the noncommutative algebra, it holds that $\db A$
is identically $0$ for a r-holomorphic element, and not only at
points where the derivative exists. This distinction becomes
important if one represents the Weyl algebra on a vector space, as in
Section \ref{sec:catenoid}, where there are elements that are not invertible.

Let us continue by defining the Laplace operator, as well as harmonic
elements and some of their properties.

\begin{definition}
  The \emph{noncommutative Laplace operator} $\Delta_0:\Fh\to\Fh$ is defined as
  \begin{align*}
    \Delta_0(A) = \dhu^2(A)+\dhv^2(A)=
    -\frac{1}{\hbar^2}\Ccom{A}{V}{V}-\frac{1}{\hbar^2}\Ccom{A}{U}{U}.
  \end{align*}
  An element $A\in\Fh$ is called \emph{harmonic} if $\Delta_0(A)=0$.
\end{definition}

\begin{proposition}
  For $A\in\Fh$ it holds that $\Delta_0(A) = 4\d\db(A)=4\db\d(A)$.
\end{proposition}

\begin{proof}
  \noindent Let us prove that $\Delta_0(A)=4\d\paraa{\db(A)}$; the
  second equality then follows from the fact that $\d\db=\db\d$. One
  computes
  \begin{align*}
    4\d\paraa{\db(A)} &= \dhu\paraa{\dhu(A)+i\dhv(A)}
    -i\dhv\paraa{\dhu(A)+i\dhv(A)}\\
    &=\dhu^2(A)+\dhv^2(A)+i\dhu\paraa{\dhv(A)}-i\dhv\paraa{\dhu(A)}\\
    &=\dhu^2(A)+\dhv^2(A) = \Delta_0(A),
  \end{align*}
  by using Proposition \ref{prop:dudvProperties}.
\end{proof}

\begin{proposition}
  Let $A\in\Fh$ be r-holomorphic. Then $\Re A$ and $\Im A$ fulfill
  \begin{align*}
    \dhu\Re A = \dhv\Im A\quad\textrm{ and }\quad
    \dhv\Re A = -\dhu\Im A,
  \end{align*}
  and it follows that $\Re A$ and $\Im A$ are harmonic.
\end{proposition}

\begin{proof}
  Since $A$ is r-holomorphic, it holds that $\db A=0$, which is
  equivalent to
  \begin{align*}
    0=\paraa{\dhu+i\dhv}\paraa{\Re A+i\Im A}=
    \dhu\Re A-\dhv\Im A+i\paraa{\dhu\Im A+\dhv\Re A}.
  \end{align*}
  Since $\Re A$ and $\Im A$ are hermitian, it follows that
  \begin{align*}
    &\dhu\Re A - \dhv\Im A=0\\
    &\dhv\Re A + \dhu\Im A = 0,
  \end{align*}
  which proves the first statement. Moreover, it is then easy to see
  that
  \begin{align*}
    \dhu^2\Re A+\dhv^2\Re A = \dhu\dhv\Im A-\dhv\dhu\Im A = 0,
  \end{align*}
  since $\dhu$ and $\dhv$ commute, by Proposition
  \ref{prop:dudvProperties}. A similar computation is done to show
  that $\Im A$ is harmonic.
\end{proof}

\noindent Integration of r-holomorphic elements is introduced as
the inverse of the operator $\d$; namely, if $A$ and $B$ are
r-holomorphic elements, such that $\d B=A$, then we call $B$ a
\emph{primitive element of $A$}. Furthermore, we introduce the
notation
\begin{align*}
  \int Ad\Lambda
\end{align*}
to denote an arbitrary primitive element of $A$. Such r-holomorphic
elements $A$, which have at least one primitive element, are called
\emph{integrable}.  Clearly, holomorphic elements, being polynomials
in $\Lambda$, are integrable, and primitive elements may readily be found.

\section{Noncommutative minimal surfaces}\label{sec:nms}

\noindent We shall consider the free module $\Fh^n$ together with its canonical
basis
\begin{align*}
  e_k=(\underbrace{0,\ldots,0}_{k-1},\mid,0,\ldots,0)
\end{align*}
and one extends the action of $\dh_a$ as
\begin{align*}
  &\dh_a(\Xv) = \dh_a(X^i)e_i
\end{align*}
for $\Xv=X^ie_i$ and $a=1,2$.  An element $\Xv\in\Fh^n$ is called
\emph{hermitian} if $X^i$ is hermitian for $i=1,\ldots,n$, and an
element $\Xv\in\Fh^n$ is called (r-)holomorphic if $X^i$ is
(r-)holomorphic for $i=1,\ldots,n$.  Moreover, for
$\Xv,\Yv\in\Fh^n$ one introduces a symmetric bi-$\complex$-linear form
\begin{align*}
  \ip{\Xv,\Yv} =\sum_{i=1}^n\ip{X^i,Y^i}\equiv\frac{1}{2}\sum_{i=1}^n
  \paraa{X^iY^i+Y^iX^i}.
\end{align*}
The above form fulfills the following derivation property, with
respect to $\dh_1$ and $\dh_2$:

\begin{proposition}
  For $\Xv,\Yv\in\Fh^n$, with $\Xv=X^ie_i$ and $\Yv=Y^ie_i$, it holds
  that
  \begin{align*}
    [\ip{\Xv,\Yv},A] = \ip{[X^i,A]e_i,\Yv}+\ip{\Xv,[Y^i,A]e_i}
  \end{align*}
  for any $A\in\Fh$. In particular, it holds that
  \begin{align*}
    &\dh_a\ip{\Xv,\Yv} = \ip{\dh_a\Xv,\Yv}+\ip{\Xv,\dh_a\Yv},
  \end{align*}
  for $a=1,2$.
\end{proposition}

\begin{proof}
  From the derivation property of the commutator it follows that
  \begin{align*}
    [AB+BA,C] = A[B,C]+[B,C]A+B[A,C]+[A,C]B,
  \end{align*}
  which may be written as
  \begin{align}\label{eq:ipderivation}
    [\ip{A,B},C] = \ip{A,[B,C]}+\ip{B,[A,C]}.
  \end{align}
  Since \eqref{eq:ipderivation} is linear in $A$ and $B$, the desired
  result follows.
\end{proof}

\noindent We will now introduce noncommutative minimal surfaces in
$\Fh^n$; this is done in analogy with the formulation in conformal
coordinates, as given in Section \ref{sec:PoissonMinimal}. It turns
out that most of the classical theory can be transferred to the
noncommutative setting with essentially no, or only small, modifications.

\begin{definition}
  A hermitian element $\Xv\in\Fh^n$ is called a \emph{noncommutative minimal
    surface} if
  \begin{align*}
    &\Delta_0(X^i) = 0\qquad\textrm{for }i=1,2,\ldots,n\\
    &\E=\G\textrm{ and }\F=0,
  \end{align*}
  where
  \begin{align*}
    \E =\ip{\dhu\Xv,\dhu\Xv},\quad
    \G =\ip{\dhv\Xv,\dhv\Xv},\quad
    \F =\ip{\dhu\Xv,\dhv\Xv}.
  \end{align*}
\end{definition}

\begin{remark}
  Note that the above definition does, in principle, not rely on the
  fraction field $\Fh$, and is also valid in the Weyl algebra $\Ah$. In
  fact, several results, in what follows, remain true in the Weyl
  algebra when r-holomorphic elements are replaced by holomorphic
  elements. We shall comment on this possibility as we proceed and
  develop the theory.
\end{remark}

\noindent Let us now define $\Phi\in\Fh^n$ as
\begin{align*}
  \Phi = \Phi^ie_i = 2\d(X^i)e_i = \paraa{\dhu(X^i)-i\dhv(X^i)}e_i
\end{align*}
and prove the following:
\begin{proposition}\label{prop:computePhiPhi}
  It holds that
  \begin{align*}
    \ip{\Phi,\Phi} = \E-\G-2i\F.
  \end{align*}
\end{proposition}

\begin{proof}
  One computes
  \begin{align*}
    \paraa{\Phi^i}^2 &=\paraa{\dhu(X^i)-i\dhv(X^i)}\paraa{\dhu(X^i)-i\dhv(X^i)}\\
    &=\dhu(X^i)^2-\dhv(X^i)^2-i\dhu(X^i)\dhv(X^i)-i\dhv(X^i)\dhu(X^i),
  \end{align*}
  which implies that
  \begin{align*}
    \sum_{i=1}^n\paraa{\Phi^i}^2 &=
    \sum_{i=1}^n\dhu(X^i)^2
    -\sum_{i=1}^n\dhv(X^i)^2\\
    &\qquad-2i\sum_{i=1}^n\frac{1}{2}\parab{\dhu(X^i)\dhv(X^i)+\dhv(X^i)\dhu(X^i)}\\
    &=\E-\G-2i\F,
  \end{align*}
  which is the desired result.
\end{proof}

\begin{proposition}\label{prop:PhiSqEFG}
  $\ip{\Phi,\Phi}=0$ if and only if $\E=\G$ and $\F=0$.
\end{proposition}
\begin{proof}
  Clearly, if $\E=\G$ and $\F=0$ then Proposition
  \ref{prop:computePhiPhi} gives $\ip{\Phi, \Phi}=0$. Now, assume that
  $\ip{\Phi, \Phi}=0$. Since $\E,\F,\G$ are hermitian, the
  $\ast$-conjugate of the equation $\ip{\Phi, \Phi}=0$ (via
  Proposition \ref{prop:computePhiPhi}) gives $\E-\G+2i\F=0$ which,
  together with $\E-\G-2i\F=0$ implies that $\E=\G$ and $\F=0$.
\end{proof}

\begin{proposition}\label{prop:harmonicEquivMinimal}
  Assume that $\Xv\in\Fh^n$ is hermitian and set
  $\Phi=2\d(\Xv)$. Then the following are equivalent
  \begin{enumerate}
  \item $\Xv$ is a minimal surface,
  \item $\Phi$ is r-holomorphic and $\ip{\Phi,\Phi}=0$.
  \end{enumerate}
\end{proposition}

\begin{proof}
  First, assume that $\Xv$ is a minimal surface (which directly
  implies, by Proposition \ref{prop:PhiSqEFG}, that $\ip{\Phi,\Phi}=0$). By
  definition, it holds that $\Delta_0(X^i)=0$, and one computes
  \begin{align*}
    0 = \Delta_0(X^i) = 4\db\paraa{\d(X^i)} 
    = 2\db(\Phi^i),
  \end{align*}
  which proves that $\Phi$ is r-holomorphic. For the other
  implication, assume that $\Phi$
  is r-holomorphic and that $\ip{\Phi,\Phi}=0$. From Proposition
  \ref{prop:PhiSqEFG} it follows that $\E=\G$ and $\F=0$. Moreover, since
  $\Phi^i$ is r-holomorphic one gets
  \begin{align*}
    0 = \db(\Phi^i) = 2\db\paraa{\d(X^i)} = \frac{1}{2}\Delta_0(X^i).
  \end{align*}
  Hence, $\Xv$ is a minimal surface.
\end{proof}

\noindent Note that the theorem remains true if $\Xv\in\Ah^n$ and
$\Phi$ is a assumed to be holomorphic. Hence, the equivalence also
holds in the Weyl algebra.

One may straightforwardly define conjugate minimal
surfaces; namely, we will call a hermitian $\Xt\in\Fh^n$ conjugate
to the minimal surface $\Xv\in\Fh^n$ if
\begin{align*}
    \dhu(\Xv) = \dhv(\Xt)\quad\text{and}\quad
    \dhv(\Xv) = -\dhu(\Xt).
\end{align*}

\begin{proposition}
  Let $\Xv\in\Fh^n$ be a minimal surface. If a hermitian
  $\Xt\in\Fh^n$ satisfies
  \begin{align*}
    \dhu(\Xv) = \dhv(\Xt)\quad\text{and}\quad
    \dhv(\Xv) = -\dhu(\Xt),
  \end{align*}
  then $\Xt$ is a minimal surface.
\end{proposition}

\begin{proof}
  One computes
  \begin{align*}
    \Delta_0(\Xt^i) &= \dhu\paraa{\dhu(\Xt^i)}+\dhv\paraa{\dhv(\Xt^i)}\\
    &=-\dhu\paraa{\dhv(X^i)}+\dhv\paraa{\dhu(X^i)}=0,
  \end{align*}
  by using Proposition \ref{prop:dudvProperties}. Moreover, it holds that
  \begin{align*}
    \tilde{\E} &= \sum_{i=1}^n\dhu(\Xt^i)^2 = \sum_{i=1}^n\dhv(X^i)^2 =
    \G,\\
    \tilde{\G} &= \sum_{i=1}^n\dhv(\Xt^i)^2 = \sum_{i=1}^n\dhu(X^i)^2
    = \E\quad(=\G=\tilde{\E}),\\
    \tilde{\F} &= \frac{1}{2}\sum_{i=1}^n
    \parab{\dhu(\Xt^i)\dhv(\Xt^i)+\dhv(\Xt^i)\dhu(\Xt^i)}\\
    &=-\frac{1}{2}\sum_{i=1}^n
    \parab{\dhv(X^i)\dhu(X^i)+\dhu(X^i)\dhv(X^i)} = -\F = 0,
  \end{align*}
  since $\Xv$ is assumed to be a minimal surface. Hence,
  $\Xt$ is a minimal surface.
\end{proof}

\subsection{Noncommutative Weierstrass representation}

\noindent The classical theory of minimal surfaces in $\reals^3$ is an
old and very rich subject. For such minimal surfaces, there are
several representation formulas available; i.e. explicit formulas for
the parametrization of an arbitrary minimal surface (see
e.g. \cite{dhko:minimalsurfacesI}). It turns out that one can prove
analogous statements in the noncommutative setting.

\begin{proposition}\label{prop:PhiHolomorhpicfg}
  Assume that $\Phi\in\Fh^3$ is r-holomorphic, fulfilling
  $\ip{\Phi,\Phi}=0$ and $\Phi^1-i\Phi^2\neq 0$. Then there exist
  r-holomorphic $f,g\in\Fh$ such that
  \begin{align*}
    \Phi^1 = \frac{1}{2}f\paraa{\mid-g^2},\quad
    \Phi^2 = \frac{i}{2}f\paraa{\mid+g^2},\quad
    \Phi^3 = fg.
  \end{align*}
  Moreover, if $\Phi$ is holomorphic then $f$ can be chosen to be holomorphic.
\end{proposition}

\begin{proof}
  First, since $\Phi^1,\Phi^2,\Phi^3$ are r-holomorphic, they
  commute; thus, one need not be careful with the ordering in what
  follows. If one sets
  \begin{align*}
    f &= \Phi^1-i\Phi^2\\
    g &= \Phi^3(\Phi^1-i\Phi^2)^{-1}
  \end{align*}
  then $f$ and $g$ are r-holomorphic (since
  $\Phi^1-i\Phi^2\neq 0)$, and one computes
  \begin{align*}
    -fg^2 &= -\paraa{\Phi^3}^2(\Phi^1-i\Phi^2)^{-1}
    =\Phi^1+i\Phi^2
  \end{align*}
  where the last equality follows from $\ip{\Phi,\Phi}=0$ (written in
  the form $(\Phi^1+i\Phi^2)(\Phi^1-i\Phi^2)+(\Phi^3)^2=0$). Now, from
  $f=\Phi^1-i\Phi^2$ and $-fg^2=\Phi^1+i\Phi^2$, the desired
  expressions for $\Phi^1$, $\Phi^2$ and $\Phi^3$ follow. Finally, we
  note that if $\Phi$ is holomorphic, then clearly $f=\Phi^1-i\Phi^2$
  is holomorphic.
\end{proof}

\noindent As a corollary we get an analogue of the Weierstrass
representation theorem.

\begin{theorem}\label{thm:WeierstrassRep}
  Let $\Xv=X^ie_i\in\Fh^3$ be a minimal surface for which it holds
  that $\d(X^1-iX^2)\neq 0$. Then there exist r-holomorphic
  elements $f,g\in\Fh$ together with $x^i\in\reals$ (for $i=1,2,3$),
  such that
  \begin{align}
    \begin{split}\label{eq:WeierstrassRep}
      X^1 &= x^1\mid + \Re\int
      \frac{1}{2}f(\mid-g^2)d\Lambda\\
      X^2 &= x^2\mid + \Re\int
      \frac{i}{2}f(\mid+g^2)d\Lambda\\
      X^3 &= x^3\mid + \Re\int
      fgd\Lambda.    
    \end{split}
  \end{align}
  Conversely, for any r-holomorphic $f$ and $g$ such that
  $f(1-g^2)$, $f(1+g^2)$ and $fg$ are integrable,
  equation \eqref{eq:WeierstrassRep} defines a minimal surface.
\end{theorem}

\begin{proof}
  Assume that $\Xv$ is a minimal surface. Setting $\Phi=2\d\Xv$ it
  follows from Proposition \ref{prop:harmonicEquivMinimal} that $\Phi$
  is r-holomorphic and $\ip{\Phi,\Phi}=0$. The assumption
  $\d(X^1-iX^2)\neq 0$ is equivalent to $\Phi^1-i\Phi^2\neq
  0$. Therefore, Proposition \ref{prop:PhiHolomorhpicfg} gives the
  existence of r-holomorphic $f$ and $g$ such that 
  \begin{align*}
    \Phi^1 = \frac{1}{2}f\paraa{\mid-g^2},\quad
    \Phi^2 = \frac{i}{2}f\paraa{\mid+g^2},\quad
    \Phi^3 = fg.
  \end{align*}
  These equations may be integrated as in \eqref{eq:WeierstrassRep},
  and since $\d\Re(A)=\d A/2$ when $A$ is r-holomorphic, they
  satisfy $\Phi=2\d\Xv$. Now, assume that $f$ and $g$ are
  r-holomorphic and that the integrals in
  \eqref{eq:WeierstrassRep} are defined. It is easy to check that
  \eqref{eq:WeierstrassRep} gives r-holomorphic $\Phi=2\d\Xv$ such that
  $\ip{\Phi,\Phi}=0$. From Proposition \ref{prop:harmonicEquivMinimal} it
  follows that $\Xv$ is a minimal surface.
\end{proof}

\noindent There is another classical representation formula, which
assigns a minimal surface to an arbitrary holomorphic function $F$.
The theorem below does not rely on r-holomorphic elements, and
therefore also holds in the Weyl algebra when $F$ is chosen
to be holomorphic.

\begin{theorem}\label{thm:WeierstrassF}
  Let $F\in\Fh$ be r-holomorphic and assume that
  \begin{align*}
    \Phi^1 = \paraa{1-\Lambda^2}F,
    \quad \Phi^2 = i\paraa{1+\Lambda^2}F,\quad
    \Phi^3 = 2\Lambda F
  \end{align*}
  are integrable. Then $\Xv=X^ie_i\in\Fh^3$, defined by
  \begin{align*}
    X^i = x^i\mid + \Re\int\Phi^id\Lambda,
  \end{align*}
  is a minimal surface for arbitrary $x^1,x^2,x^3\in\reals$.
\end{theorem}

\begin{proof}
  By definition, $X$ is hermitian, and one computes that
  \begin{align*}
    2\d(X^i) &= \d\int\Phi^id\Lambda +
    \d\parac{\parab{\int\Phi^id\Lambda}^\ast}\\
    &=\d\int\Phi^id\Lambda  = \Phi^i,
  \end{align*}
  since $\d$ applied to a quotient of polynomials in $\Lambda^\ast$
  gives zero. Moreover, a simple computation shows that
  $\ip{\Phi,\Phi}=0$ for every r-holomorphic
  $F\in\Fh$. Finally, since $\Phi^i$ is r-holomorphic, it follows
  from Proposition \ref{prop:harmonicEquivMinimal} that $\Xv$ is a
  minimal surface.
\end{proof}

\noindent In the geometric setting, a minimal surface constructed via Theorem
\ref{thm:WeierstrassF} has a normal vector given by
\begin{align*}
  N = \frac{1}{1+u^2+v^2}\paraa{2u,2v,u^2+v^2-1}.
\end{align*}
Let us show that, with respect to the symmetric form
$\ip{\cdot,\cdot}$, a noncommutative normal can be constructed.

\begin{proposition}
  Let $\Xv\in\Fh^3$ be a minimal surface given by an r-holomorphic
  element $F\in\Fh$, as in Theorem \ref{thm:WeierstrassF}. Then
  $\Nv=N^ie_i\in\Fh^3$, given by
  \begin{align*}
    N^1 = \Lambda+\Lambda^\ast,\quad
    N^2 = -i(\Lambda-\Lambda^\ast),\quad
    N^3 = \frac{1}{2}\paraa{\Lambda\Lambda^\ast + \Lambda^\ast\Lambda}-\mid
  \end{align*}
  satisfies $\ip{\dhu\Xv,\Nv}=\ip{\dhv\Xv,\Nv} = 0$.
\end{proposition}

\begin{proof}
  The proof consists of a straightforward computation. The statement
  that $\ip{\dhu\Xv,\Nv}=\ip{\dhv\Xv,\Nv} = 0$ is equivalent to
  $\ip{\d\Xv,\Nv}=\ip{\db\Xv,\Nv}=0$, which in turn is equivalent to
  $\ip{\Phi,\Nv}=\ip{\Phi^\ast,\Nv}=0$. Since $\Nv$ is hermitian, it
  is enough to prove that $\ip{\Phi,\Nv}=0$. With $\Phi$ as in Theorem
  \ref{thm:WeierstrassF}, one computes that
  \begin{align*}
    N^1\Phi^1+N^2\Phi^2+N^3\Phi^3 &=
    \Lambda\Lambda^\ast\Lambda F-\Lambda^\ast\Lambda^2
    F\\
    &=[\Lambda,\Lambda^\ast]\Lambda F=2\hbar\Lambda F,
  \end{align*}
  as well as
  \begin{align*}
    \Phi^1N^1+\Phi^2N^2+\Phi^3N^3 &=
    F\Lambda\Lambda^\ast\Lambda-F\Lambda^2\Lambda^\ast\\
    &=F\Lambda[\Lambda^\ast,\Lambda] = -2\hbar F\Lambda,
  \end{align*}
  which implies that $\ip{\Phi,\Nv}=0$.
\end{proof}

\noindent Let us end this section by noting that the ``mean curvature'' of a
minimal surface vanishes. As in differential geometry, given a normal
element $\Nv\in\Fh^3$, one may define
the mean curvature (in conformal coordinates) as 
\begin{align*}
  H(\Nv) = -\frac{1}{2\E}\ip{\dhu\Xv,\dhu\Nv}
  -\frac{1}{2\E}\ip{\dhv\Xv,\dhv\Nv}\equiv
  \frac{1}{\E}H_0(\Nv).
\end{align*}
Hence, if $\Delta_0(\Xv)=0$ then it follows that
\begin{align*}
  2H_0(\Nv) &=
  \ip{\dhu^2\Xv,\Nv}-\dhu\ip{\dhu\Xv,\Nv}
  +\ip{\dhv^2\Xv,\Nv}-\dhv\ip{\dhv\Xv,\Nv}\\
  &= \ip{\Delta_0(\Xv),\Nv} = 0.
\end{align*}
Conversely, if $H_0(\Nv)=0$ then $\ip{\Delta_0{\Xv},\Nv}=0$, and it
follows from $\E=\G$ and $\F=0$ that
\begin{align*}
  \ip{\dhu\Xv,\Delta_0(\Xv)}=\ip{\dhv\Xv,\Delta_0(\Xv)}=0.
\end{align*}
However, since $\ip{\cdot,\cdot}$ is not $\Fh$-linear, these equations
do not necessarily imply that $\Delta_0(\Xv)=0$.

\section{Examples}

\subsection{Algebraic minimal surfaces}

\noindent A holomorphic element $F$ may be integrated an
arbitrary number of times. Hence, choosing a holomorphic element
$\Ft$ such that $\d^3\Ft = F$, the representation formula in
Theorem \ref{thm:WeierstrassF} may be integrated (via partial
integration) to yield
\begin{equation}\label{eq:algebraicWeierstrass}
  \begin{split}
    &X^1 = x^1\mid +
    \Re\parab{(\mid-\Lambda^2)\d^2\Ft+2\Lambda\d\Ft-2\Ft}\equiv
    x^1\mid +\Re\paraa{\Omega^1}\\
    &X^2 = x^2\mid +
    \Re\parab{i(\mid+\Lambda^2)\d^2\Ft-2i\Lambda\d\Ft+2i\Ft}\equiv
    x^2\mid +\Re\paraa{\Omega^2}\\
    &X^3 = x^3\mid + 
    \Re\parab{2\Lambda\d^2\Ft-2\d\Ft}\equiv
    x^3\mid +\Re\paraa{\Omega^3}.
  \end{split}
\end{equation}
In other words, every holomorphic $\Ft(\Lambda)$ gives rise to a minimal
surface via \eqref{eq:algebraicWeierstrass}. As an example, let us
choose $\Ft(\Lambda)=\Lambda^n$ (with $n\geq 2$), which gives
\begin{align*}
  &\Omega^1 = (n-1)\parab{n\Lambda^{n-2}-(n-2)\Lambda^n}\\
  &\Omega^2 = i(n-1)\parab{n\Lambda^{n-2}+(n-2)\Lambda^n}\\
  &\Omega^3 = 2n(n-2)\Lambda^{n-1}.
\end{align*}
We note that the real part of $\Lambda^n$ consists of the total
symmetrization of all monomials with an even (total) power of
$V$. That is,
\begin{align*}
  \Re(\Lambda^n) = \sum_{k=0}^{\intpart{\frac{n}{2}}}
  (-1)^k\Sym(U^{n-2k}V^{2k}),
\end{align*}
where $\Sym(U^kV^l)$ denotes the sum of all terms of different
permutations of $k$ $U$'s and $l$ $V$'s, and $\intpart{r}$ denotes the
integer part of $r\in\reals$. Likewise, it holds that
\begin{align*}
  \Re(i\Lambda^n) = \sum_{k=1}^{\intpart{\frac{n+1}{2}}}(-1)^k
  \Sym\paraa{U^{n-2k+1}V^{2k-1}},
\end{align*}
and one obtains the following explicit representation formulas
\begin{align*}
  &X^1 = x^1\mid + \sum_{k=0}^{\intpart{\frac{n-2}{2}}}(-1)^k
  \Sym\paraa{U^{n-2(k+1)}V^{2k}}-\frac{n-2}{n}\sum_{k=0}^{\intpart{\frac{n}{2}}}(-1)^k
  \Sym\paraa{U^{n-2k}V^{2k}}\\
  &X^2 = x^2\mid + \sum_{k=1}^{\intpart{\frac{n-1}{2}}}(-1)^k
  \Sym\paraa{U^{n-1-2k}V^{2k-1}}\\
  &\qquad\qquad\qquad+\frac{n-2}{n}\sum_{k=1}^{\intpart{\frac{n+1}{2}}}(-1)^k
  \Sym\paraa{U^{n-2k+1}V^{2k-1}}\\
  &X^3 = x^3\mid + 
  \frac{2(n-2)}{n-1}\sum_{k=0}^{\intpart{\frac{n-1}{2}}}(-1)^k
  \Sym\paraa{U^{n-1-2k}V^{2k}}.
\end{align*}
Thus, one may construct a noncommutative minimal surface from
the classical one by completely symmetrizing the polynomials. As an
illustration, let us consider the first two non-trivial minimal surfaces arising
in this way. 

For $\Ft(\Lambda)=\Lambda^3$ (corresponding to $F(\Lambda)=6$) one
obtains the noncommutative Enneper surface
\begin{align*}
  &X^1 = x^1\mid + U-\frac{1}{3}U^3+\frac{1}{3}\Sym\paraa{UV^2}\\
  &X^2 = x^2\mid - V +
  \frac{1}{3}V^3-\frac{1}{3}\Sym\paraa{U^2V}\\
  &X^2 = x^3\mid + U^2 - V^2
\end{align*}
which, using that $[U,V]=i\hbar\mid$, can be written as
\begin{align*}
  &X^1 = x^1\mid + U + UV^2-\frac{1}{3}U^3-i\hbar V\\
  &X^2 = x^2\mid -V -U^2V+\frac{1}{3}V^3+i\hbar U\\
  &X^3 = x^3\mid + U^2 - V^2.
\end{align*}
For $\Ft(\Lambda)=\Lambda^4$ (corresponding to $F(\Lambda)=24\Lambda$)
one obtains
\begin{align*}
  &X^1 = x^1\mid+U^2-V^2-\frac{1}{2}\paraa{U^4+V^4}
  +\frac{1}{2}\Sym\paraa{U^2V^2}\\
  &X^2 = x^2\mid-UV-VU-\frac{1}{2}\Sym(U^3V)+\frac{1}{2}\Sym(UV^3)\\
  &X^3 = x^3\mid + \frac{4}{3}U^3-\frac{4}{3}\Sym(UV^2),
\end{align*}
which may be written as
\begin{align*}
  &X^1 = \parac{x^1-\frac{3}{2}\hbar^2}\mid + U^2-V^2
  -\frac{1}{2}\paraa{U^4+V^4}+3U^2V^2
  -6i\hbar U\\
  &X^2 = (x^2+i\hbar)\mid - 2UV - 2U^3V + 2UV^3
  -3i\hbar V^2+3i\hbar U^2\\
  &X^3 = x^3\mid + \frac{4}{3}U^3-4UV^2+4i\hbar V.
\end{align*}
Algebraic surfaces can also be obtained from Theorem
\ref{thm:WeierstrassRep}, some of which cannot be constructed as in
Theorem \ref{thm:WeierstrassF}. For instance, choosing $f=2$ and
$g=\Lambda^n$ gives the higher order Enneper surfaces as
\begin{align*}
  X^1&=x^1\mid+U-\frac{1}{2n+1}\sum_{k=0}^n(-1)^k\Sym\paraa{U^{2n+1-2k}V^{2k}}\\
  X^2&=x^2\mid-V+\frac{1}{2n+1}\sum_{k=1}^{n+1}(-1)^k\Sym\paraa{U^{2n+2-2k}V^{2k-1}}\\
  X^3&=x^3\mid+\frac{2}{n+1}\sum_{k=0}^{\intpart{\frac{n+1}{2}}}(-1)^k
  \Sym\paraa{U^{n+1-2k}V^{2k}},
\end{align*}
which, for $n=2$, becomes
\begin{align*}
  X^1 &=x^1\mid+U+2U^3V^2-UV^4-\frac{1}{5}U^5
  -6i\hbar U^2V+2i\hbar V^3-3\hbar^2U\\
  X^2 &=x^2\mid-V+2U^2V^3-U^4V-\frac{1}{5}V^5
  -6i\hbar UV^2+2i\hbar U^3-3\hbar^2V\\
  X^3 &=x^3\mid-2UV^2+\frac{2}{3}U^3+2i\hbar V.
\end{align*}

\subsection{Minimal surfaces in $\Fh^4$}

\noindent For two holomorphic functions $f(z)$ and $g(z)$, it is well
known (cp. \cite{e:minmalEuclideanFour}) that one can construct a minimal surface in $\reals^4$ by
setting
\begin{align*}
  \xv = (\Re f(z),\Im f(z),\Re g(z),\Im g(z)).
\end{align*}
This extends to noncommutative minimal surfaces:

\begin{proposition}
  Let $f,g\in\Fh$ be r-holomorphic and set $\Xv=X^ie_i\in\Fh^4$ with
  \begin{align*}
    \paraa{X^1,X^2,X^3,X^4}=\paraa{\Re f, \Im f, \Re g, \Im g}.
  \end{align*}
  Then $\Xv$ is a minimal surface.
\end{proposition}

\begin{proof}
  Defining $\Phi=2\d\Xv$ yields
  \begin{align*}
    \paraa{\Phi^1,\Phi^2,\Phi^3,\Phi^4} = 
    \paraa{\d f,-i\d f,\d g,-i\d g},
  \end{align*}
  which implies that $\ip{\Phi,\Phi}=0$. From Proposition
  \ref{prop:harmonicEquivMinimal} it follows that $\Xv$ is a minimal
  surface (since $\Phi^i$ is clearly r-holomorphic).
\end{proof}

\noindent As an example, let us choose $f(\Lambda)=\Lambda^{n}$ and
$g(\Lambda)=\Lambda^m$, which implies that
\begin{align*}
  &X^1 = \Re(\Lambda^n) = \sum_{k=0}^{\intpart{\frac{n}{2}}}
  (-1)^k\Sym\paraa{U^{n-2k}V^{2k}}\\
  &X^2 = \Im(\Lambda^n) = \sum_{k=0}^{\intpart{\frac{n+1}{2}}}
  (-1)^{k+1}\Sym\paraa{U^{n-2k+1}V^{2k-1}}\\
  &X^3 = \Re(\Lambda^m) = \sum_{k=0}^{\intpart{\frac{m}{2}}}
  (-1)^k\Sym\paraa{U^{m-2k}V^{2k}}\\
  &X^4 = \Im(\Lambda^m) = \sum_{k=0}^{\intpart{\frac{m+1}{2}}}
  (-1)^{k+1}\Sym\paraa{U^{m-2k+1}V^{2k-1}},
\end{align*}
and for $f=\Lambda$ and $g=\Lambda^2$ one obtains
\begin{align*}
  (X^1,X^2,X^3,X^4) = (U,V,U^2-V^2,2UV-i\hbar\mid).
\end{align*}

\subsection{Noncommutative catenoids}\label{sec:catenoid}

\noindent The minimal surfaces in the preceding section are algebraic
in the sense that they arise from (finite) polynomials. Several
classical minimal surfaces, such as the catenoid, are constructed in
terms of analytic functions, which are, a priori, not defined in the
algebra. However, as we shall see, one may construct particular
representations in which certain power series are well defined. (A
different approach to the catenoid was taken in \cite{ah:quantizedMinimal}.) 

Let $\V$ be the vector space consisting of infinite sequences of
complex numbers
\begin{align*}
  \V = \{(x_0,x_1,x_2,\ldots):x_i\in\complex\text{ for }i\in\naturals_0\},
\end{align*}
and we denote the canonical basis vectors by $\ket{n}$,
$n\in\naturals_0$. For convenience, we shall write an element
$x=(x_0,x_1,x_2,\ldots)\in \V$ as a formal sum
\begin{align*}
  x = \sum_{k=0}^\infty x_k\ket{k}.
\end{align*}
The space of linear operators $\V\to\V$ is denoted by
$L(\V)$. Moreover, we introduce the subspace $\V_0\subset \V$ of
finite linear combinations
\begin{align*}
  \V_0 = \{x\in \V: |i: x_i\neq 0|<\infty\},
\end{align*}
and denote the set of linear operators with domain $\V_0$ by $L(\V_0,\V)$.
As is well known, the Weyl algebra can be represented on $\V$ by
introducing operators $a,\ad\in L(\V)$, defined by
\begin{align*}
  &a\ket{0} = 0\\
  &a\ket{n} =\sqrt{n}\ket{n-1}\text{ for }n\geq 1\\
  &\ad\ket{n}=\sqrt{n+1}\ket{n+1},
\end{align*}
fulfilling $[a,\ad]\ket{n}=\ket{n}$, and then setting
\begin{align*}
  &U = \sqrt{\frac{\hbar}{2}}\paraa{\ad+a}\\
  &V = i\sqrt{\frac{\hbar}{2}}\paraa{\ad-a},
\end{align*}
from which it follows that $\Lambda=U+iV=\sqrt{2\hbar}a$ and
$\Ld\equiv\Lambda^\ast = U-iV=\sqrt{2\hbar}\ad$.  We note that the
operators $U$ and $V$ leave the subspace $\V_0$ invariant. Let us
recall two useful formulas:
\begin{lemma}
  \begin{align}
    &a^k\ket{n} =
    \begin{cases}
      \sqrt{\frac{n!}{(n-k)!}}\ket{n-k}&\text{ if }k\leq n\\
      0&\text{ if }k>n
    \end{cases}\\
    &(\ad)^k\ket{n} = \sqrt{\frac{(n+k)!}{n!}}\ket{n+k}.
  \end{align}
\end{lemma}

\noindent For arbitrary $\lambda\in\complex$ we define linear operators
$e^{\lambda a},e^{\lambda\ad}\in L(\V_0,\V)$ as
\begin{align*}
  &e^{\lambda a}\ket{n} = 
  \sum_{k=0}^\infty \frac{(\lambda a)^k}{k!}\ket{n}
  =\sum_{k=0}^n\frac{\lambda^k}{k!}\sqrt{\frac{n!}{(n-k)!}}\ket{n-k}\\
  &e^{\lambda\ad}\ket{n} =
  \sum_{k=0}^\infty\frac{(\lambda\ad)^k}{k!}\ket{n}
  =\sum_{k=0}^\infty\frac{\lambda^k}{k!}\sqrt{\frac{(n+k)!}{n!}}\ket{n+k}.
\end{align*}
Furthermore, let us introduce $\dhu,\dhv,\d,\db,\Delta_0:L(\V)\to
L(\V)$, defined via commutators, as in Section
\ref{sec:weylAlgebra}. Since $U$ and $V$ leave $\V_0$ invariant, the
aforementioned maps can be considered as maps $L(\V_0,V)\to
L(\V_0,V)$.

The classical catenoid may be parametrized as ($z=u+iv$)
\begin{align*}
  &x^1(u,v) = \Re(\cosh z) = \cosh u\cos v\\
  &x^2(u,v) = \Re(-i\sinh z) = \cosh u\sin v\\
  &x^3(u,v) = \Re(z) = u
\end{align*}
arising from the Weierstrass data $f(z)=-e^{-z}$ and $g(z)=-e^z$
(cp. Theorem \ref{thm:WeierstrassRep})\footnote{Note that there exist
  other possibilities for $f$ and $g$.}. In analogy, we set
\begin{align*}
  &X^1 =
  \frac{1}{4}\parab{e^\Lambda+e^{-\Lambda}+e^{\Ld}+e^{-\Ld}}\\
  &X^2 =
  -\frac{i}{4}\parab{e^{\Lambda}-e^{-\Lambda}-e^{\Ld}+e^{-\Ld}}\\
  &X^3 = U
\end{align*}
which implies that $X^1,X^2,X^3\in L(\V_0,\V)$; we will now show that
$\Delta_0(X^i)=0$ for $i=1,2,3$. 
\begin{lemma}
  For $\lambda\in\complex$, it holds that
  \begin{align}
    &[e^{\lambda a},\ad]\ket{n} = \lambda e^{\lambda a}\ket{n}\\
    &[e^{\lambda\ad},a]\ket{n} = -\lambda e^{\lambda\ad}\ket{n}.
  \end{align}
\end{lemma}
\noindent From the above result, one easily deduces 
\begin{align*}
  &\d e^{\lambda\Lambda}\ket{n} =
  \lambda e^{\lambda\Lambda}\ket{n}\qquad\qquad
  \db e^{\lambda\Lambda}\ket{n} = 0\\
  &\db e^{\lambda\Ld}\ket{n} = 
  \lambda e^{\lambda\Ld}\ket{n}\qquad\quad
  \d e^{\lambda\Ld}\ket{n} = 0
\end{align*}
for arbitrary $\lambda\in\complex$. Since $\Delta_0(X^i)=4\db\d X^i$ one
obtains
\begin{align*}
  &\Delta_0(X^1)\ket{n} = \db\parab{e^\Lambda-e^{-\Lambda}}\ket{n} = 0\\
  &\Delta_0(X^2)\ket{n} = -i\db\parab{e^\Lambda+e^{-\Lambda}}\ket{n} = 0\\
  &\Delta_0(X^3)\ket{n} = 2\db\d\paraa{\Lambda+\Ld}\ket{n} = 
  2\db(\mid)\ket{n} = 0.
\end{align*}
Hence, $\Delta_0{X^i}$, for $i=1,2,3$, are $0$ as operators in $L(\V_0,\V)$.

What about the condition that the parametrization is conformal? That
is
\begin{align*}
  \ip{\dhu\Xv,\dhu\Xv}=\ip{\dhv\Xv,\dhv\Xv}\text{ and }
  \ip{\dhu\Xv,\dhv\Xv} = 0.
\end{align*}
Since $X^1$ and $X^2$ do not preserve $\V_0$, their composition is a
priori not well defined. However, algebraically, the above
is equivalent to $\ip{\Phi,\Phi}=0$ (cp. Proposition \ref{prop:PhiSqEFG}); with
\begin{align*}
  &\Phi^1\ket{n} = 2\d X^1\ket{n}
  = \frac{1}{2}\parab{e^{\Lambda}-e^{-\Lambda}}\ket{n}\\
  &\Phi^2\ket{n} = 2\d X^2\ket{n}
  = -\frac{i}{2}\parab{e^{\Lambda}+e^{-\Lambda}}\ket{n}\\
  &\Phi^3\ket{n} = U\ket{n}
\end{align*}
the expression $\ip{\Phi,\Phi}$ is well defined, since
$e^{\pm\Lambda}$ maps $\V_0$ into $\V_0$, and one readily checks that
$\ip{\Phi,\Phi}=0$.

\bibliographystyle{alpha}
\bibliography{nms}  

\end{document}